\documentclass[12]{amsart}
\usepackage[T1]{fontenc}
\usepackage[francais, english]{babel}
\usepackage{amsmath,amssymb,hyperref}
\usepackage{graphics,color}
\usepackage[all]{xy}
\usepackage{geometry} 
\usepackage{array,multirow,makecell}
\newtheorem{thm}{Theorem}

\newtheorem{cor}[thm]{Corollary}
\newtheorem{prop}[thm]{Proposition}
\newtheorem{lem}[thm]{Lemma}
\newtheorem{definition}[thm]{Definition}
\newtheorem{remark}[thm]{Remark}

\newcommand{\OM}[1]{\Omega^1_{#1}}

\newcommand{\App}{\mathrm{App}} 
\newcommand{\pp}{\zeta} 

\DeclareMathOperator{\tr}{tr}

\DeclareMathOperator{\diag}{diag}
\DeclareMathOperator{\Res}{Res}

\def\Z{\mathbb Z}

\def\C{\mathbb C}
\def\P{\mathbb P}

\def\O{\mathcal O}

\def\U{\mathcal U}
\def\V{\mathcal V}
\def\W{\mathcal W}

\def\p{\boldsymbol{p}}

\def\Mu{\boldsymbol{\mu}}

\def\Con{\mathcal{M}}

\def\OMD{\mathbf{\Omega}_D}
\def\OM0{\mathbf{\Omega}_0}

\begin{document}

\author[K. Diarra]{Karamoko DIARRA}
\address{ DER de Math\'ematiques et d'informatique, FAST, Universit\'e des Sciences, des Techniques et des Technologies de Bamako, BP: E $3206$ Mali.}
\email{karamoko.diarra2005@gmail.com}
 
\author[F. Loray]{Frank LORAY}
\address{ Univ Rennes, CNRS, IRMAR - UMR 6625, F-35000 Rennes, France.}
\email{frank.loray@univ-rennes1.fr}

\title[Normal form]{Normal forms for rank two linear irregular differential equations and moduli spaces.}

\date{\bf \today}
\thanks{We thank CNRS, Universit\'e de Rennes 1, Henri Lebesgue Center, ANR-16-CE40-0008 project ``{\it Foliage}'' for financial support and CAPES-COFECUB Ma932/19 project.}

\begin{abstract} We provide a unique normal form for rank two irregular connections on the Riemann sphere.
In fact, we provide a birational model where we introduce apparent singular points and where the bundle
has a fixed Birkhoff-Grothendieck decomposition. The essential poles and the apparent poles provide two
parabolic structures. The first one only depends on the formal type of the singular points. The latter one 
determines the connection (accessory parameters). As a consequence, an open set of the corresponding 
moduli space of connections is canonically identified with an open set of some Hilbert scheme of points on 
the explicit blow-up of some Hirzebruch surface. This generalizes to the irregular case previous results obtained
by Szab\'o. Our work is more generally related to ideas and descriptions of Oblezin, Dubrovin-Mazzocco,  
and Saito-Szab\'o in the logarithmic case. After the first version of this work appeared, Komyo used our
normal form to compute isomonodromic Hamiltonian systems for irregular Garnier systems.
\end{abstract}

\maketitle
\setcounter{tocdepth}{1}
\sloppy
\tableofcontents

\section{Introduction}

In this paper, we deal with rank 2 meromorphic connections on the Riemann sphere $\mathbb P^1=\mathbb P^1(\mathbb C)$,
which consist in the data $(E,\nabla)$ of a rank 2 holomorphic vector bundle $E\to \mathbb P^1$, and a meromorphic connection 
$\nabla:E\to E\otimes\Omega^1_{\mathbb P^1}(D)$, i.e. $D$ is an effective divisor (the polar divisor) on $\mathbb P^1$, and
$\nabla$ is a $\mathbb C$-linear map satisfying Leibniz rule: for any open set $U\subset\mathbb P^1$, we have
$$\nabla(f\cdot s)=df\otimes s+f\cdot \nabla(s)$$
for any holomorphic function $f$ on $U$ and section $s:U\to E$ (see section \ref{sec:LocalFormalData} for more details).
Recall that the connection is said {\it logarithmic} (or {\it Fuchsian}, or {\it regular-singular}) when it has simple poles, i.e. when 
the polar divisor $D$ is reduced. In this paper, we are particularly interested in the {\it irregular} setting, 
i.e. non logarithmic connections, when $D$ has points with multiplicity.

\subsection{Main results}
Fix local formal data $\Lambda$ (see section \ref{sec:LocalFormalData}) including the polar locus $D$.
It is the local data of a global meromorphic connection $(E,\nabla)$ provided that it satisfies Fuchs relation:
the sum of residual eigenvalues equals $-\deg(E)$, and must therefore be an integer.
The moduli space of those $\Lambda$-connections up to bundle isomorphisms can be constructed
by Geometric Invariant Theory. After fixing convenient weights $\Mu$, the moduli space $\Con_{\Mu}^\Lambda$ 
of $\Mu$-stable parabolic $\Lambda$-connections $(E,\nabla,\l)$ forms 
an irreducible quasi-projective variety 
of dimension $2(n-3)$ where $n=\deg(D)$ (see \cite{IIS,InabaSaito,Inaba}).
Moreover, $\Mu$-stable connections form a Zariski open subset which is smooth, and contains all irreducible connections.
In fact, the role of weights and parabolic data are mainly here to avoid bad singularities in the quotient,
and are useful only for special choices of $\Lambda$; for generic $\Lambda$,
the parabolic structure is automatically determined by $(E,\nabla)$ and any connection is irreducible, and thus $\Mu$-stable whatever 
is the choice of weights $\Mu$; therefore, $\Con_{\Mu}^\Lambda$ can be viewed as the moduli space of $\Lambda$-connections 
$(E,\nabla)$.
In this paper, we would like to describe a natural open set of this moduli space in an explicit way,
together with an explicit universal family.
Natural operations on parabolic connections like twisting by rank one meromorphic connections,
or applying convenient birational bundle modifications provide isomorphisms between moduli spaces,
usually called canonical transformations. In particular, this has the effect to shift $\deg(E)$ by arbitrary integers, and 
we can assume $\deg(E)=1$ without loss of generality. 
We generalize\footnote{It has been recently generalized to higher rank and genus in the logarithmic setting
by Saito and Szab\'o in \cite{SS}.}  to the irregular setting a construction appearing in a paper of Oblezin \cite{Oblezin}
 and Szab\'o \cite[Corollary 2]{Szilard} allowing to define a natural birational model for $\Con_{\Mu}^\Lambda$:
\begin{thm}\label{thm:BiratModel} 
Let $\Lambda$ be as before. Then, there is a natural birational map to the Hilbert scheme
\begin{equation}\label{eq:BiratSym}
\Phi:\Con_{\Mu}^\Lambda\dashrightarrow\mathrm{Hilb}^{(n-3)}(\OMD)
\end{equation}
where $\OMD$ denotes the total space of the line bundle $\Omega^1_{\mathbb P^1}(D)$.
\end{thm}

Recall that the Hilbert scheme $\mathrm{Hilb}^{(n-3)}(S)$ of a smooth surface $S$ is the minimal 
reduction of singularities by blow-up of the symmetric product 
$$\mathrm{Hilb}^{(n-3)}(S)\to\mathrm{Sym}^{(n-3)}(S)=\underbrace{S\times\cdots\times S}_{n-3\ \text{times}}/\sigma_{n-3}.$$ 

There is also a natural holomorphic symplectic structure on $\Con_{\Mu}^\Lambda$ (see 
\cite{Kimura,Boalch,Kawamuko,InabaSaito,Inaba})
that generalizes the symplectic structure of Atyiah-Bott, Goldman and Iwasaki in the regular(-singular) case.
On the other hand, the Liouville form on the total space $\OM0$ of the cotangent bundle $\Omega^1_{\mathbb P^1}$ induces a symplectic structure
on $\mathrm{Hilb}^{(n-3)}(\OM0)$, and therefore, via the natural map $\Omega^1_{\mathbb P^1}\to\Omega^1_{\mathbb P^1}(D)$, 
on a Zariski open set of $\mathrm{Hilb}^{(n-3)}(\OMD)$.
After the first version of this paper was put on the arXiv, Komyo used our normal form
to produce the associated isomonodromic Hamiltonian system in \cite{Komyo}; in particular, he proved

\begin{thm}[Komyo \cite{Komyo}]\label{thm:KomyoSymplectic}
The map $\Phi$ is symplectic\footnote{Up to a sign, depending on authors' conventions.}.
\end{thm}

In section \ref{sec:Scalar}, we easily check that the map $\Phi$ is symplectic in the classical Garnier case
(i.e. the logarithmic case, with simple poles) by comparing with the symplectic structure given in \cite{Okamoto,Kimura,DM}.
We also checked this in the irregular case, up to 5 poles counted with multiplicity by comparing with formulae
of Kimura and Kawamuko \cite{Kimura,Kawamuko}. 

\subsection{Apparent map and companion normalization}
In the theory of linear system of differential equations with rational coefficients,
it is often useful to reduce to a single scalar differential equation, whose order corresponds
to the rank of the initial system; this is done by choosing a {\it cyclic vector} and using it to reduce the matrix 
into a {\it companion form}, which is the matrix form of a scalar equation.  This operation (not unique in general)
introduces new poles for the differential equations, called {\it apparent singular points}. 
The construction of the map $\Phi$ in the above Theorems is inspired by ideas developped in 
\cite{Oblezin,DM,Szilard,LSS,LS,KomyoSaito,SS} to provide a geometric counterpart to this classical
operation with the cyclic vector. 
The very rough idea is to consider a connection on non trivial bundles so that there is a 
natural rank one subbundle $L\subset E$ which will play the role of a canonical cyclic vector. 
Precisely, for a generic $(E,\nabla)$ with $\deg(E)=1$, we have\footnote{From now on, 
we omit the subscript ${\mathbb P^1}$ in the line bundles for notational simplicity.}
$E\simeq\O\oplus\O(1)$ and we can choose $L=\mathcal O(1)$ (the unique destabilizing subsheaf).
In fact, the subset
\begin{equation}\label{eq:DefU}
\U=\{(E,\nabla)\in\Con_{\Mu}^\Lambda \ ;\  E\simeq\mathcal O\oplus\O(1)\ \text{and $\O(1)$ is not $\nabla$-invariant}\}
\end{equation}
is Zariski open in $\Con_{\Mu}^\Lambda$. For such a generic connection,
we associate a non trivial $\mathcal O$-linear morphism 
\begin{equation}\label{eq:sequenceApparent}\xymatrix{
L\ar@{^{(}->}[r] \ar@/_2pc/[rrrr]^-{\varphi_\nabla} & E \ar[r]^-{\nabla} & E\otimes\Omega^1(D) \ar[rr]^-{\text{quotient}} &&
(E/L)\otimes\Omega^1(D) \simeq\Omega^1(D).
}\end{equation}
The {\it apparent map} then associates the effective zero-divisor
of $\varphi_\nabla$:
\begin{equation}\label{eq:ApparentMap}
\App\ :\ \U
\to\vert\mathcal O(n-3)\vert
\ ;\ (E,\nabla)\mapsto\mathrm{div}(\varphi_{\nabla}).
\end{equation}
Then, we can define the birational bundle transformation
\begin{equation}\label{eq:BiratModif}
\phi_\nabla:=\text{id}\oplus\varphi_\nabla:\ \mathcal O\oplus\mathcal O(1)\ \dashrightarrow\ \mathcal O\oplus\Omega^1(D)
\end{equation}
and consider the pushed-forward connection $(\phi_\nabla)_*\nabla$ on $E_0=\mathcal O\oplus\Omega^1(D)$.
This operation introduces new poles located on the support of the divisor $Q=\App(E,\nabla)$; these are 
apparent singular points (local horizontal sections are holomorphic) whence the name of $\App$.
To provide precise statement, it is convenient to introduce the following Zariski open subsets
\begin{equation}\label{eq:DefW}
W=\{Q\in \vert\mathcal O(n-3)\vert\ ;\ Q\ \text{is reduced and has disjoint support with}\ D\} \\
\end{equation}
and
\begin{equation}\label{eq:DefV}
\V=\U\cap\App^{-1}(W)\subset\Con_{\Mu}^\Lambda
\end{equation}
In section \ref{sec:GlobalNormalization}, we prove the following 

\begin{prop}[Companion Form]\label{prop:DefNabla0}Assume $(E,\nabla)\in\V$.
Then, the transformed connection $(\phi_\nabla)_*\nabla$ has polar divisor $D+Q$,
and is equivalent, up to bundle automorphisms of $E_0=\mathcal O\oplus\Omega^1(D)$, 
to a unique connection $\nabla_0$ satisfying
\begin{itemize}
\item the restriction of $\nabla_0$ to the first factor $\mathcal O\subset E_0$ is the trivial connection: 
$$\nabla_0\vert_{\mathcal O}\equiv0;$$
\item the morphism $\varphi_{\nabla_0}:\Omega^1(D)\to\Omega^1(D)$ defined by the ``cyclic vector''
$\Omega^1(D)\subset E_0$ similarly as in (\ref{eq:sequenceApparent}) is the identity:
$$\varphi_{\nabla_0}=\mathrm{id}_{\Omega^1(D)}.$$
\end{itemize}
Moreover, at each point of $Q$,
the residual eigenvalues are $-1$ and $0$, and the $(-1)$-eigendirection is contained in 
the subbundle $\Omega^1(D)$.
\end{prop}

\subsection{Construction of $\Phi$}
Let $\nabla_0$ as in Proposition \ref{prop:DefNabla0}, and denote $Q=\{q_1,\ldots,q_{n-3}\}$.
The $0$-eigendirection of $\nabla_0$ at $q_j$ corresponds to a point $\p_j\in\mathbb P(E_0)$.
On the other hand, $\mathbb P(E_0)$ can be viewed as the fiber-compactification of $\OMD$
(the total space of $\Omega^1(D)$), and $\p_j$ is actually contained in that open set. 
We therefore define the map (\ref{eq:BiratSym}) by setting
\begin{equation}\label{eq:BiratSymDef}
\Phi(E_0,\nabla_0)=\{\mathbf p_1,\ldots,\mathbf p_{(n-3)}\} \in \mathrm{Hilb}^{(n-3)}(\OMD).
\end{equation}
The natural projection $\pi:\OMD\to\mathbb P^1;\mathbf p_j\mapsto q_j$ induces a morphism
$$\Pi:\mathrm{Hilb}^{(n-3)}(\OMD)\to\mathrm{Hilb}^{(n-3)}\mathbb P^1
\simeq\vert\mathcal O(n-3)\vert$$
which is Lagrangian. Set $\W=\Pi^{-1}(W)$.
We deduce from our construction of $\Phi$:

\begin{prop}\label{prop:AppIsom}
We have a commutative diagram of morphisms
\begin{equation}\label{eq:AppLagrangian}\xymatrix{\relax
\Con_{\Mu}^\Lambda\supset \V\ar[rr]^-{\Phi}_{\sim}\ar[drr]_{\App} &&\W\ar[d]^{\Pi\vert_{\W}}&\subset\mathrm{Hilb}^{(n-3)}(\OMD)\ar[d]^{\Pi}\\
&& W&{\subset\vert\mathcal O(n-3)\vert }
}\end{equation}
inducing a biregular isomorphism $\V\stackrel{\sim}{\to}\W$. 
\end{prop}

Applying Komyo's result  (Theorem \ref{thm:KomyoSymplectic}) we deduce:

\begin{cor}\label{cor:AppLagrangian}
The apparent map defines a Lagrangian fibration on the moduli space $\Con_{\Mu}^\Lambda$,
which is smooth in restriction to the open set $\V\subset\Con_{\Mu}^\Lambda$ with affine fibers $\mathbb C^{n-3}$.
\end{cor}

In fact, these $\C^{n-3}$-fibers play the role of {\it accessory parameters} in the classical theory.
In order to prove Proposition \ref{prop:AppIsom}, implying Theorem \ref{thm:BiratModel}, 
we explicitely construct the inverse application on $\Pi^{-1}(W)$. In Section \ref{Sec:PhiBirat}, we explicitely construct 
$\nabla_0$ as follows:

\begin{lem}\label{lem:UniversalNabla0}
Let $D$ and $\Lambda$ as before.
Given a divisor $Q\in W\subset \vert\mathcal O(n-3)\vert$ 
and given a lift
$$\{\mathbf p_1,\ldots,\mathbf p_{(n-3)}\}\in\Pi^{-1}(Q) \subset \mathrm{Hilb}^{(n-3)}(\OMD),$$ 
then there is a unique connection $\nabla_0$ on $E_0$ such that 
\begin{itemize}
\item $\nabla_0$ has polar divisor $D+Q$,
\item $\nabla_0$ satisfies the normalizations $\nabla_0\vert_{\mathcal O}\equiv0$ and $\varphi_{\nabla_0}=\mathrm{id}_{\Omega^1(D)}$ of Proposition \ref{prop:DefNabla0},
\item the local formal data of $\nabla_0$ at $D$ is $\Lambda$,
\item at each pole $q_j\in Q$, $\nabla_0$  is an apparent singular point with residual eigenvalues $-1$ and $0$,
and eigenspaces respectively given by $\Omega^1(D)\subset E_0$ and $\mathbf p_j$.
\end{itemize}
\end{lem}

\subsection{Consequences}
Given $(E_0,\nabla_0)$ as in the Lemma, after a birational bundle modification,
we get rid of apparent singular points and obtain a class $(E,\nabla)$ of $\Con_{\Mu}^\Lambda$.
Moreover, this can be done in family 
$$\{\mathbf p_1,\ldots,\mathbf p_{(n-3)}\}\mapsto(E_0,\nabla_0)\rightsquigarrow(E,\nabla)$$
giving rise to a {\it universal family} on $\V\subset\Con_{\Mu}^\Lambda$.

We expect that this approach to the moduli space of connections will be useful to describe the total
moduli space of connections, as this has been done for instance in \cite{KomyoSaito} 
for the logarithmic case (with a complete description in the case $n=5$). 
In fact, we prove in Corollary \ref{cor:FixedLeveltDec} that the Hukuhara-Levelt-Turrittin splitting subsheaves 
for $\nabla_0$ at each pole of $D$ 
is determined by $\Lambda$ up to order equal to the multiplicity of the pole. We think that 
this decomposition can be used to blow-up the surface $\OMD\subset S_D=\P(E_0)$
and get an embedding of the larger open set 
$$\Con_{\Mu}^\Lambda\supset \U\hookrightarrow \mathrm{Hilb}^{n-3}(\widehat{\OMD})\subset \mathrm{Hilb}^{n-3}(\widehat{S_D}).$$
Next, to recover missing connections and embed the whole of $\Con_{\Mu}^\Lambda$,
we expect that we need to perform some birational flip to $\mathrm{Hilb}^{n-3}(\widehat{S_D})$
(see for instance \cite[Theorem1.4]{KomyoSaito}). It is important to understand the compactification
of $\Con_{\Mu}^\Lambda$, for instance for the Geometric Langlands Program (see \cite{ArinkinFedorov}).
Also we expect that our approch could be useful to better understand coalescence of poles in $D$ 
and degenerescence diagram for Garnier systems. 

\subsection{Thanks}
We would like to thank very much A. Komyo, M. H. Saito and S. Szab\'o for useful discussions on the subject.
We would like also to thank the referee who helped us to improve the presentation of our results.

\section{Local formal data}\label{sec:LocalFormalData}
In this section, we recall the local classification of meromorphic connections at poles up to formal gauge tranformations.
This is a very particular case of the theory developped by Hukuhara (1942), Turritin (1955) and Levelt (1975); we refer to 
\cite{Varadarajan} and \cite[Sections 15 \& 21]{IY} for more details and references therein.

 Consider a connection $(E,\nabla)$ with polar divisor $D=\sum_{i=1}^\nu n_i\cdot [t_i]$ on $\mathbb P^1$,
i.e. $t_i$ are distinct points and $n_i\in\mathbb Z_{>0}$ are multiplicities.
On any strict open set $U\subset \mathbb P^1$, the vector bundle $E$ is trivial,
i.e. admits global coordinates $Y\in\mathbb C^2$, and in such trivialization, the connection writes
$$\nabla:\ Y\mapsto dY+\Omega\cdot Y,$$
where $\Omega$ is the matrix-connection, a two-by-two matrix of meromorphic $1$-forms, sections of $\Omega^1(D)$.
In a coordinate $x:U\to \mathbb C$, we have 
$$\Omega=\begin{pmatrix}a(x) & b(x)\\ c(x) & d(x)\end{pmatrix}\frac{dx}{P(x)}$$
where $P=\prod_{i=1}^\nu(x-t_i)^{n_i}$ is the monic polynomial with divisor $\mathrm{div}(P)=D\vert_U$
and $a,b,c,d$ are holomorphic functions on $U$ (not all of them vanishing). 

Let $z$ be a local coordinate at a pole of $\nabla$, and $Y$ a local trivialization of $E$.
Then we can write $\nabla:\ Y\mapsto dY+\Omega\cdot Y$ with
\begin{equation}\label{Eq:LocalOmegaIrreg}
\Omega=\left(\underbrace{\frac{A_0}{z^{k+1}}+\frac{A_1}{z^k}+\cdots+\frac{A_k}{z}}_{\left[\Omega\right]^{<0}_{z=0}}+A(z)\right) dz
\end{equation}
where $n:=k+1$ is the order of pole, $A_0,A_1,\ldots,A_k\in\mathrm{gl}_n(\mathbb C)$ are constant matrices,
$A_0\not=0$, and $A\in\mathrm{gl}_n(\mathbb C\{z\})$ is holomorphic. We note $\left[\Omega\right]^{<0}_{z=0}$ the negative (or principal)
part of $\Omega$ at $z=0$. 

Given $M\in \mathrm{GL}_n(\mathbb C[[z]])$, the formal change of variable $Y=M\tilde Y$, that we usually call {\bf gauge transformation},
gives the new matrix connection 
$$\nabla:\tilde Y\mapsto d\tilde Y + \underbrace{\left( M^{-1}\Omega M+M^{-1}dM\right)}_{\tilde\Omega}\tilde Y.$$
We note that the new negative part only depends on the negative part of $\Omega$:
$$\left[\tilde\Omega\right]^{<0}_{z=0}=\left[M^{-1}\left[\Omega\right]^{<0}_{z=0} M\right]^{<0}_{z=0}$$
and is obtained just by conjugacy of the negative part of $\Omega$ by $M$  (since $M^{-1}dM$ is holomorphic).

\begin{remark}\label{rem:reduced}
We can also use formal meromorphic gauge transformations $M\in \mathrm{GL}_n(\mathbb C((z)))$.
Usually, this is used to obtain to decrease the pole order (when possible) of $\Omega$,
in order to obtain a good formal model.
In this paper, we do not consider meromorphic gauge transformation, and we will assume 
that the pole order is minimal up to such transformation. 
Another usual transformation is the twist (or tensor product) by a rank one holomorphic or meromorphic connection.
In local trivializations, this consists in adding a scalar matrix of meromorphic $1$-forms to $\Omega$.
For instance, when $A_0$ is not a scalar matrix, i.e. $A_0\not=c\cdot I$, then tensoring
by the rank $1$ connection $d-c\frac{dz}{z^{k+1}}$ would kill the matrix coefficient of $\frac{dz}{z^{k+1}}$.
Again, we will not use such transformations in this paper, and assume that $A_0$ is not scalar, so that 
the pole order cannot decrease by a meromorphic twist.
\end{remark}

\begin{definition}The connection $\nabla$ is said {\bf reduced} at $z=0$ if the pole order of $\nabla$ cannot decrease
by applying a meromorphic gauge transformation and/or twist by a meromorphic rank one connection, as in the previous remark.
\end{definition}

From now on, we will only consider connections with reduced poles, in particular assuming $A_0$ non scalar in (\ref{Eq:LocalOmegaIrreg}),
and we will only use (formal) holomorphic gauge transformations. In Proposition \ref{prop:ConsFormInv},
we will give an easy criterium for a pole to be reduced. 

\subsection{Logarithmic case $n=1$ (and $k=0$)} Clearly, the eigenvalues $\theta_1,\theta_2\in\mathbb C$ of $A_0$ 
are invariant under formal gauge transformation $M$ as above, and it is known (see \cite[Section 15]{IY}) that there exists  
a holomorphic $M$ such that 
\begin{equation}\label{eq:FormeNormLocLog}
\tilde\Omega=\begin{pmatrix}\theta_1&0\\ 0& \theta_2\end{pmatrix}\frac{dz}{z}\ \ \ \text{or}\ \ \ 
\begin{pmatrix}\theta+n& x^n\\ 0& \theta\end{pmatrix}\frac{dz}{z}
\end{equation}
in the resonant case $\{\theta_1,\theta_2\}=\{\theta,\theta+n\}$, with $n\in\mathbb Z_{\ge0}$.
In the construction of moduli spaces of connections, it is natural to fix the spectral data $\{\theta_1,\theta_2\}$
rather than the formal type (see \cite{IIS}). 

\subsection{Irregular $n>1$ and unramified case}\label{ss:irrun} Assume that the leading coefficient $A_0$ is semi-simple: it has distinct eigenvalues (recall $A_0$ cannot be scalar). 
\begin{prop}\label{prop:FormalUnram}
If $A_0$ has distinct eigenvalues in (\ref{Eq:LocalOmegaIrreg}), 
then there exists a formal $M$ such that
\begin{equation}\label{eq:FormeNormLocUnram}
\tilde\Omega=M^{-1}\Omega M+M^{-1}dM=\begin{pmatrix}\lambda^+&0\\ 0& \lambda^-\end{pmatrix}
\end{equation}
where
$$\lambda^{\pm}=\left(\frac{\theta_0^{\pm}}{z^{k+1}}+\frac{\theta_1^{\pm}}{z^{k}}+\cdots+\frac{\theta_k^{\pm}}{z}\right)dz.$$
Moreover, $M$ is unique up to right multiplication by a constant diagonal or anti-diagonal matrix.
\end{prop}

We note that the formal invariants $\lambda^{\pm}$ can be determined from the negative part $\left[\Omega\right]^{<0}_{z=0}$
of the initial matrix: it suffices to solve 
$$\det\bigg([\Omega]^{<0}_{z=0}-\lambda I\bigg)=\lambda^2-\mathrm{tr}\bigg([\Omega]^{<0}_{z=0}\bigg)\lambda+\det\bigg([\Omega]^{<0}_{z=0}\bigg)=0$$
up to the $k+1$ first coefficients.

An easy consequence of this normal form is that the connection admits a unique formal decomposition 
as product of rank $1$ connections. Indeed, one can write $E=L^+\oplus L^-$ with formal line bundles 
$L^{\pm}\subset E$ that are invariant by $\nabla$, and $\nabla\vert_{L^{\pm}}=d+\lambda^{\pm}$.
This is the Hukuhara-Levelt-Turrittin splitting in its simplest form (see Varadarajan \cite{Varadarajan}).

\begin{proof} 
We first diagonalize the connection as follows.
Assume first 
$$\Omega=\begin{pmatrix}\alpha&\beta\\ \gamma&\delta\end{pmatrix}=\left(\frac{A_0}{z^{k+1}}+\cdots\right)dz\ \ \ 
\text{with}\ \ \ 
A_0=\begin{pmatrix}\theta_0^+&0\\0&\theta_0^-\end{pmatrix}$$ 
by a preliminary change by a constant matrix $M$. 
Then we start solving $\det(\Omega-\lambda I)=0$
and because $\theta_0^{\pm}$ are distinct, then there are exactly two meromorphic solutions
$$\lambda^{\pm}=\left(\frac{\theta_0^{\pm}}{z^{k+1}}+\cdots\right)dz.$$
Then, the invertible holomorphic matrix 
$$M=\begin{pmatrix}1&\frac{\beta}{\lambda^--\alpha}\\ \frac{\gamma}{\lambda^+-\delta}&1\end{pmatrix}$$
satisfies $M^{-1}\Omega M=\diag(\lambda^+,\lambda^-)$. As $M^{-1}dM$ is holomorphic, then the gauge transformation
$\tilde\Omega=M^{-1}\Omega M+M^{-1}dM$ has diagonal principal part. If we now start with $\Omega$ 
diagonal up to some positive order $m$, i.e. antidiagonal terms up to $z^m$ vanish, then $M$ will be tangent to the 
identity up to order $m+k+1$, and $\tilde\Omega$ will be diagonal up to order $m+k+1$. By iterating the above process,
we find a change $M=M_1M_2M_3\cdots$ which converges in the formal holomorphic setting, and the resulting gauge
transformation diagonalizes the connection. We note that eigenvalues $\lambda^{\pm}$ of $\Omega$ change at each step,
but their principal part does not. Next, by a diagonal holomorphic change $M$, which commute with the diagonal connection 
$\Omega=M^{-1}\Omega M$, we get $\tilde\Omega=M^{-1}dM$, and it is easy to find such a $M$ killing the positive part 
of $\Omega$. This latter transformation is unique up to a constant diagonal matrix. 
One easily check that the total transformation $M$ is also well-defined in the formal holomorphic setting.
The unique invariant line bundles $L^{\pm}$ are generated then by the standart basis. 
\end{proof}


\subsection{Irregular $n>1$ and ramified case}\label{ss:irram}
We now assume that the leading coefficient $A_0$ is a non trivial Jordan block: it has a single eigenvalue with multiplicity $2$ 
but is not scalar, i.e. we can assume
after a preliminary constant gauge transformation $M$ that
\begin{equation}\label{eq:CondRamif}
\Omega=\left(\frac{\begin{pmatrix}a_0&b_0\\0&a_0\end{pmatrix}}{z^{k+1}}+\frac{\begin{pmatrix}a_1&b_1\\c_1&d_1\end{pmatrix}}{z^k}+\cdots\right)dz\ \ \ \text{with}\ b_0\not=0.
\end{equation}
If $c_1=0$ in the form above,
i.e. the $(2,1)$-coefficient of $A_1$ also vanish, then the birational gauge transformation 
$M=\diag(1,z)$ transforms $A_0$ into a scalar matrix $\diag(a_0,a_0)$, so that one can decrease the pole order 
by twisting with the rank one connection $d-a_0\frac{dz}{z^{k+1}}$. It is therefore natural to assume that 
$c_1\not=0$, and it is possible to check that the pole order cannot decrease anymore by gauge or twist. 

\begin{prop}\label{prop:LocalNormalRamified}
If the connection takes the form (\ref{eq:CondRamif}) with $c_1\not=0$, then 
 there exists a formal $M$ such that 
 \begin{equation}\label{eq:FormeNormLocRam}
\tilde\Omega=M^{-1}\Omega M+M^{-1}dM=\begin{pmatrix}\alpha&\beta \\ z\beta & \alpha-\frac{dz}{2z}\end{pmatrix}
\end{equation}
where
$$\left\{\begin{matrix}
\alpha=&\left(\frac{a_0}{z^{k+1}}+\frac{a_1}{z^{k}}+\cdots\cdots\cdots+\frac{a_k}{z}\right)dz \\
\beta=&\left(\frac{b_0}{z^{k+1}}+\frac{b_1}{z^{k}}+\cdots+\frac{b_{k-1}}{z^2}\right)dz\hfill\hfill
\end{matrix}\right. $$
Moreover, $M$ is unique up to composition by matrices of the form
$$\begin{pmatrix}a&cz\\ c&a\end{pmatrix}\ \ \ \text{or}\ \ \ \begin{pmatrix}a&cz\\ -c&-a\end{pmatrix},\ \ \ a,c\in\mathbb C;$$
the first type preseves the normal form, while the second one changes the sign $\beta\mapsto -\beta$.
\end{prop}
We have one less formal invariant, namely $2k+1$ instead of $2k+2$.
After ramification $z=\zeta^2$, and applying the birational gauge transformation 
$$\tilde Y=\begin{pmatrix}1&1\\ \zeta&-\zeta\end{pmatrix}\bar Y,$$ 
we get an unramified irregular singular point with matrix connection 
$$\bar\Omega=\begin{pmatrix}\lambda^+&0 \\ 0 & \lambda^-\end{pmatrix}$$
where
$$\lambda^\pm=\alpha\pm \zeta\beta.$$
Therefore, again we have a unique formal decomposition $E=L^+\oplus L^-$,
but in a ramified variable $\zeta=\sqrt{z}$, i.e. $L^\pm$ is generated by the 
``section'' 
$$\zeta\mapsto\begin{pmatrix}1\\ \pm\zeta\end{pmatrix}.$$

\begin{proof}[Proof of Proposition \ref{prop:LocalNormalRamified}]
Applying the base change $z=\zeta^2$ in (\ref{eq:CondRamif}), and the elementary transformation 
$$Y=\begin{pmatrix}1&0\\ 0&\zeta\end{pmatrix}\bar Y$$ 
to the connection $d+\Omega$ yields 
$$\bar\Omega=2\left(\frac{\begin{pmatrix}a_0&0\\0&a_0\end{pmatrix}}{\zeta^{2k+1}}+\frac{\begin{pmatrix}0&b_0\\c_1&0\end{pmatrix}}{\zeta^{2k}}+\cdots\right)d\zeta\ \ \ \text{with}\ b_0\not=0.
$$
We can delete the dominant term, which is scalar, by a twist and get back to the irregular unramified case.
We observe that the lifted connection is invariant under the combination of $\zeta\mapsto-\zeta$ with the 
gauge transformation $M=\diag(1,-1)$. This transformation permutes the two formal invariant line bundles
$L^+$ and $L^-$ of section \ref{ss:irrun}, so that we get a ramified invariant line bundle  for $d+\Omega$
generated by a multisection of the form 
$$z\mapsto\begin{pmatrix}1\\ s(\sqrt{z})\end{pmatrix}\ \ \ \text{with}\ \ \ s(\zeta)=\sqrt{b_0c_1}\zeta+o(\zeta^2)\in\mathbb C[[\zeta]].$$
After a formal gauge transformation $Y=M\tilde Y$, we can normalize this ramified bundle to that one generated by
$$z\mapsto\begin{pmatrix}1\\ \sqrt{z}\end{pmatrix}$$
and it is easy to check that the resulting matrix connection $\tilde\Omega$ is in the normal form of the statement.
The last assertion follows from the shape of matrices preserving the above ramified line bundle.
\end{proof}

\subsection{Local formal data}\label{ss:LocalConsequences}
We can rephrase the formal classification in the irregular case as follows. Consider the equation 
\begin{equation}
\det(\Omega-\lambda I)=\lambda^2-\mathrm{tr}(\Omega)\lambda+\det(\Omega)=0\ \ \ \text{where}\ 
\Omega=\begin{pmatrix}\alpha&\beta\\ \gamma&\delta\end{pmatrix}
\end{equation}
and its discriminant:
\begin{equation}
\Delta(\Omega)=\mathrm{tr}(\Omega)^2-4\det(\Omega)=(\alpha-\delta)^2+4\beta\gamma.
\end{equation}
($\det(\Omega)$ and $\Delta(\Omega)$ are meromorphic quadratic differentials). If $n$ is the order of pole of $\Omega$
at $z=0$, then the first $n$ coefficients of the negative part
$$[\mathrm{tr}(\Omega)]^{<0}_{z=0},\ \ \ [\det(\Omega)]^{<-n}_{z=0},\ \ \ \text{and}\ \ \ [\Delta(\Omega)]^{<-n}_{z=0}$$
only depend on the negative part $\left[\Omega\right]^{<0}_{z=0}$ of $\Omega$, up to formal holomorphic gauge transformation.

\begin{prop}\label{prop:ConsFormInv}
The pole order $n>1$ of the connection $d+\Omega$ is reduced  if, and only if, 
the pole order of $\Delta(\Omega)$ is $2n$ or $2n-1$; moreover, these two possible orders respectively correspond 
to the unramified and ramified cases. The negative part of $\mathrm{tr}(\Omega)$ and $\det(\Omega)$ provide 
a full set of invariants for the formal classification, i.e. up to formal holomorphic gauge transformation:
$$\Omega'=M^{-1}\Omega M+M^{-1}dM\ \ \ \Leftrightarrow\ \ \ \left\{\begin{matrix}
[\mathrm{tr}(\Omega')]^{<0}_{z=0}=[\mathrm{tr}(\Omega)]^{<0}_{z=0}\\
[\det(\Omega')]^{<-n}_{z=0}=[\det(\Omega)]^{<-n}_{z=0}
\end{matrix}\right.$$
\end{prop}

\begin{proof}
This can be easily deduced from the discussions of the previous subsections. In the unramified case,
the formal invariants $\lambda^{\pm}$ of Proposition \ref{prop:FormalUnram} are given by the negative 
part $[\lambda]^{<0}_{z=0}$
of the two solutions of 
$$\lambda^2-\big([\mathrm{tr}(\Omega)]^{<0}_{z=0}\big)\lambda+\big([\det(\Omega)]^{<-n}_{z=0}\big)=0.$$
In the ramified case, the formal invariants $\alpha$ and $\pm\beta$ of Proposition \ref{eq:FormeNormLocRam}
are given by first solving 
$$2\alpha-\frac{dz}{2z}=[\mathrm{tr}(\Omega)]^{<0}_{z=0}\ \ \ \text{and}\ \ \ 
\alpha^2-\alpha\frac{dz}{2z}-z\beta^2=[\det(\Omega)]^{<-n}_{z=0}$$
and then take the negative part of the formal meromorphic solutions $\alpha$ and $\pm\beta$.
\end{proof}

We therefore conclude that the equivalence class of the connection $d+\Omega$ up to formal 
holomorphic gauge transformation is characterized in the irregular case by the negative parts
$[\mathrm{tr}(\Omega)]^{<0}_{z=0}$ and $[\det(\Omega)]^{<-n}_{z=0}$ of usual invariants.
In the logarithmic case, it correspond to the usual spectral data of the singular point
(but fails however to characterize up formal gauge transformation in the resonant case).

\begin{prop}\label{prop:FormalInvCoordChange}
The negative parts of linear or quadratic differentials at $t_0=\{z=0\}$
$$\tr(\nabla)^{<0}_{t_0}:=[\mathrm{tr}(\Omega)]^{<0}_{z=0}\ \ \ \text{and}\ \ \ 
\det(\nabla)^{<-n}_{t_0}:=[\det(\Omega)]^{<-n}_{z=0}$$
are well defined by $(E,\nabla)$ as section germs at $t_0$ of $\Omega^1(D)/\Omega^1$ 
and $(\Omega^1(D))^{\otimes 2}/(\Omega^1)^{\otimes 2}(D)$ respectively, 
independantly of the choice of the coordinate $z$. We denote this pair of invariants
by 
$$\Lambda_{t_0}(\nabla)=\left( \tr(\nabla)^{<0}_{t_0}\ ,\ \det(\nabla)^{<-n}_{t_0} \right)$$
\end{prop}

\begin{proof}If we apply a change of coordinate $z'=\varphi(z)$, then the matrix connection is changed as follows
$$\Omega'=\varphi^*\Omega$$
where $\varphi^*$ is the pull-back on differential forms applied coefficient-wise. Therefore,
the trace part is changed as follows:
$$[\mathrm{tr}(\Omega')]^{<0}_{z'=0}=[\varphi^*\mathrm{tr}(\Omega)]^{<0}_{z'=0}
=\left[\varphi^*[\mathrm{tr}(\Omega)]^{<0}_{z=0}\right]^{<0}_{z'=0}.$$
We have a similar formula for the determinant.
\end{proof}

\begin{definition}We call formal data of $(E,\nabla)$, and denote by $\Lambda$, the collection of 
formal data $\Lambda_{t_0}$ at each pole $t_0$ of $\nabla$.
\end{definition}

We recall the Fuchs relation which follows from residue formula:

\begin{prop}\label{prop:Fuchs}
If $(E,\nabla)$ is a rank 2 meromorphic connection on a (smooth irreducible) curve $C$,
with poles $t_1,\ldots,t_\nu$ and formal data $\Lambda$, then we have 
\begin{equation}\label{eq:Fuchs}
\sum_{i=1}^\nu\mathrm{Res}_{t_i}\left(\tr(\nabla)^{<0}_{t_i}\right)\ +\ \det(E)=0.
\end{equation}
\end{prop}

\section{Explicit normalization}\label{sec:GlobalNormalization}

We now want to consider global rank 2 connections $(E,\nabla)$ on $\mathbb P^1$.
Denote by $x$ a coordinate on an affine chart of $\mathbb P^1$ and by $x=\infty$ the deleted point.
The vector bundle $\mathcal O\oplus\mathcal O(k)$ is defined by two charts, 
one with coordinates $(x,Y)\in\C\times\C^2$, and another one with coordinates $(\tilde x,\tilde Y)$ satisfying
$$x=1/\tilde x,\ \ \ Y=\begin{pmatrix}1&0\\ 0 & x^k\end{pmatrix}\tilde Y.$$
For $k>0$, the automorphism group of this vector bundle consists of transformations of the form
$$Y\mapsto \begin{pmatrix}u&0\\ F & v\end{pmatrix} Y$$
where $u,v\in\mathbb C^*$ and $F(x)$ is a polynomial of degree $\deg(F)\le k$.
A connection on $\mathcal O\oplus\mathcal O(k)$ takes the form $d+\Omega$ and $d+\tilde\Omega$ in the respective charts with
\begin{equation}\label{eq:SystemAtInfinity}
\Omega=\begin{pmatrix}a(x) & b(x)\\ c(x) & d(x)\end{pmatrix}dx\ \ \ \text{and}\ \ \ 
\tilde\Omega=\begin{pmatrix}-\tilde x^{-2}a(\tilde x^{-1}) & -\tilde x^{-k-2}b(\tilde x^{-1})\\ 
-\tilde x^{k-2}c(\tilde x^{-1}) & -\tilde x^{-2}d(\tilde x^{-1})-k\tilde x^{-1}\end{pmatrix}d\tilde x
\end{equation}
where $a,b,c,d$ are rational on $\P^1$.

Let us now fix the divisor of poles for $\nabla$: it can be written into the form
$$D=\sum_{i=1}^\nu n_i\cdot [t_i]+n_\infty\cdot[\infty]\ \ \ \text{and}\ \ \ n:=\deg(D)=\sum_{i=1}^\nu n_i + n_\infty$$
with $n_i\in\Z_{>0}$, $i=1,\ldots,\nu,\infty$. For simplicity, we will assume $n_\infty>$, i.e. that $\infty$ is indeed
a pole of the connections. In view of proving Proposition \ref{prop:DefNabla0},
let us assume that $E=\mathcal O\oplus\mathcal O(1)$. Then the matrix connection for $\nabla$ takes the form
$$\Omega=\begin{pmatrix}\frac{A(x)}{P(x)} & \frac{B(x)}{P(x)} \\ \frac{C(x)}{P(x)} & \frac{D(x)}{P(x)}\end{pmatrix}dx$$
where $P(x)=\prod_{i=1}^\nu(x-t_i)^{n_i}$ and $A,B,C,D$ are polynomials of degree 
\begin{equation}\label{eq:ConnectionAtInfinity}
\left\{\begin{matrix}
\deg(B)&\le&n-3\\
\deg(A),\deg(D)&\le& n-2\\
\deg(C)&\le& n-1
\end{matrix}\right.
\end{equation}
with at least one inequality which is an equality.
In fact, if $(E,\nabla)\in\U$ (see (\ref{eq:DefU})), then we know that $B\not\equiv0$ 
since $\O(1)$ is not $\nabla$-invariant. Moreover,
the apparent map $\varphi_\nabla$ (see (\ref{eq:sequenceApparent})) is defined by $B(x)$,
and the apparent divisor is
$$Q=\{B(x)=0\}\in \vert\mathcal O(n-3)\vert.$$
From now on, we assume that $(E,\nabla)\in\V$ (see (\ref{eq:DefV})), like in Proposition \ref{prop:DefNabla0},
which means that $B$ has degree $n-3$, with simple roots $q_1,\cdots,q_{(n-3)}$,
all of them distinct from the poles $t_i$'s.
Now, define the birational bundle transformation (see (\ref{eq:BiratModif}))
$$\phi_{\nabla}: \mathcal O\oplus\mathcal O(1)\dashrightarrow E_0=\mathcal O\oplus\mathcal O(n-2)$$
in matrix form by setting 
$$Y_1=MY\ \ \ \text{with}\ \ \ M=\begin{pmatrix}1&0\\ 0 & B(x)\end{pmatrix}.$$
We obtain the new matrix connection
$$\Omega_1= \begin{pmatrix}\frac{A(x)}{P(x)} & \frac{1}{P(x)} \\ \frac{B(x)C(x)}{P(x)} & \frac{D(x)}{P(x)}-\sum_{i=1}^{n-3}\frac{1}{x-q_i}\end{pmatrix}dx$$
and the change of chart is now given by 
$$\tilde x=1/x,\ \ \ \tilde Y_1=\begin{pmatrix}1&0\\ 0 & x^{n-2}\end{pmatrix}Y_1.$$
Geometrically, $\phi_\nabla$ is obtained by applying a positive elementary transformation directed by $\mathcal O(1)$
over each point $x=q_i$.
Finally, we normalize by applying a holomorphic bundle automorphism. Automorphisms of $E_0$ take the form
$$Y_0=MY_1\ \ \ \text{with}\ \ \ M=\begin{pmatrix}u&0\\ F & v\end{pmatrix}$$
where $u,v\in\C^*$ are constants and $F(x)$ is a polynomial of degree $\le n-2$.
By choosing $F=-A$, we can kill the $(1,1)$-coefficient of the matrix 
\begin{equation}\label{eq:Omega2}
\Omega_0= \begin{pmatrix}0 & \frac{1}{P(x)} \\ c_0(x) & d_0(x)
\end{pmatrix}dx.
\end{equation}
Geometrically, this means that {\it we can deform the subbundle $\mathcal O$ in a unique way such that $\nabla\vert_{\mathcal O}$ is the trivial 
connection} ! We note that the apparent map $\varphi_{\nabla_0}$ is not equal $1$;
moreover, the remaining freedom is given by  scalar automorphisms
(with $u=v$ and $F=0$), which commute with everything. This normalization is therefore unique
and we have proven the first part of Proposition \ref{prop:DefNabla0}.
Now, notice that the residual part of $\Omega_1$ at $x=q_j$ is given by 
$$\Res_{x=q_j}\Omega_1=\begin{pmatrix} 0&0\\ 0&-1\end{pmatrix}$$
so that residual eigenvalues are $-1$ and $0$, with $(-1)$-eigendirection contained in 
the subbundle $\Omega^1(D)$. Since automorphisms of $E_0$ preserve this subbundle, then $\Omega_0$ 
has the same property, therefore ending the proof of Proposition \ref{prop:DefNabla0}.
\hfill$\square$

\begin{prop}After the series of transformations above, we get the unique normal form (\ref{eq:Omega2})
and the coefficients take the form
\begin{equation}\label{eq:GlobalNormalForm}
\left\{\begin{matrix}c_0(x)=\sum_{i=1}^{\nu}\frac{C_i(x)}{(x-t_i)^{n_i}}+ \sum_{j=1}^{n-3}\frac{\pp_j}{x-q_j}+\tilde C(x)+x^{n-3}C_\infty(x) \\
d_0(x)=\sum_{i=1}^{\nu}\frac{D_i(x)}{(x-t_i)^{n_i}}+\sum_{j=1}^{n-3}\frac{-1}{x-q_j}+D_\infty(x)\hfill\hfill \end{matrix}\right.
\end{equation}
where
\begin{itemize}
\item $\deg(C_i),\deg(D_i)\le n_i-1$ for $i=1,\ldots,\nu$,
\item $\deg(C_\infty)\le n_\infty-1$ and $\deg(D_\infty)\le n_\infty-2$,
\item $\deg(\tilde C)\le n-4$.
\end{itemize}
\end{prop}

As we shall see, polynomials $C_i,D_i$, for $i=1,\ldots,\nu,\infty$, are determined by the formal data $\Lambda$
(see Lemma \ref{lem:CiDi}), and $\tilde C$ by the fact that $q_j$ are apparent singular points (see Lemma \ref{Lem:tildeC}).

\begin{proof}Coefficients $c_0$ and $d_0$ of (\ref{eq:Omega2}) have pole equation $P(x)B(x)=\prod_i(x-t_i)^{n_i}\prod_j(x-q_j)=0$.
Taking into account the pole order at $x=\infty$, 
formula (\ref{eq:SystemAtInfinity}) with $k=n-2$ yields 
$$\deg(c_0)=n+n_\infty-4\ \ \ \text{and}\ \ \ \deg(d_0)=n_\infty-2.$$
On the other hand, the principal part of $c_0$, say, along the polar divisor $D+Q$ write
$$\sum_{i=1}^{\nu}\frac{C_i(x)}{(x-t_i)^{n_i}}+ \sum_{j=1}^{n-3}\frac{\pp_j}{x-q_j}$$
so that the difference should be a polynomial of degree $n+n_\infty-4$. We split is as 
$\tilde C(x)+x^{n-3}C_\infty(x)$ in order to underline the contribution $C_\infty(x)$
to the principal part at $x=\infty$. The decomposition of $d_0$ is similar, except 
that we know that it has residues $-1$ along $Q$.
\end{proof}

The $0$-eigendirection of $\nabla_0$ over $x=q_j$ is defined by $\begin{pmatrix}1\\ \pp_j\end{pmatrix}$
so that we have canonically associated the collection of points $\{\mathbf p_1,\ldots,\mathbf p_{(n-3)}\}$ 
where $\mathbf p_j=(q_j,\pp_j)$ stands for the point of 
$\OMD$, the total space of $\Omega^1(D)$, defined by the section $\pp_j\cdot\frac{dx}{P(x)}$ at $x=q_j$.
The map $\Phi$ of Proposition \ref{prop:AppIsom} is therefore explicitely defined as follows. 
We first define $\Phi'$  on the open set 
$\V$ by 
$$\Phi'(E,\nabla)=\{\mathbf p_1,\ldots,\mathbf p_{(n-3)}\}\in\mathrm{Sym}^{(n-3)}(\OMD).$$
Recall that $H:\mathrm{Hilb}^{(n-3)}(\OMD)\to\mathrm{Sym}^{(n-3)}(\OMD)$ is a birational morphism
(see \cite{Gottsche}) which consists of blowing-up along the big diagonal. The complement 
$\W'$ of the diagonal is exactly the image of $\W$, in restriction to which $H$ becomes a biregular isomorphism.
Then we can lift $\Phi':\V\to\W'$ as a map $\Phi:\V\to\W$ such that $H\circ\Phi=\Phi'$. 
We therefore obtain a commutative diagram\vskip1cm
$$\xymatrix{
\Con_{\Mu}^\Lambda \ar@{.>}[r] \ar@/^2pc/[rr]^{\Phi} & \mathrm{Sym}^{(n-3)}(\OMD)  & \mathrm{Hilb}^{(n-3)}(\OMD) \ar[l]\\
\V \ar@{^{(}->}[u] \ar[r]^{\sim} \ar@/^2pc/[rr]^{\sim} \ar[dr]^{\App} & \W' \ar@{^{(}->}[u] \ar[d]^q & \W \ar[l]_{\sim} \ar@{^{(}->}[u] \ar[ld]^\Pi \\
& W & \subset \vert\mathcal O(n-3)\vert
}$$
which proves the first part of Proposition \ref{prop:AppIsom}. In order to end-up the proof,
it remains to show that $\Phi$ is birational, and in particular inducing  a biregular isomorphism 
$\Phi:\V\stackrel{\sim}{\to}\W$. In the next section, we construct an inverse of $\Phi$ on $\W$.

\section{Birationality of $\Phi$}\label{Sec:PhiBirat}

In this section, we want to reverse the map $\Phi$ on a Zariski open set.
Namely, we now want to prove that, knowing local formal data $\Lambda$,
and knowing the position $\{\mathbf p_1,\ldots,\mathbf p_{(n-3)}\}$ of $0$-eigendirections of apparent singular points,
we can uniquely determine the matrix connection $\Omega_0$ defined by (\ref{eq:Omega2}) and (\ref{eq:GlobalNormalForm}), i.e. the connection $(E_0,\nabla_0)$. 
If true, we can recover $(E,\nabla)$ up to isomorphism by applying an elementary transformation
at the $0$-eigendirection of each apparent pole $q_j$.

\begin{lem}\label{lem:CiDi}
The polynomials $C_i,D_i$ occuring in formula (\ref{eq:GlobalNormalForm}) can be determined 
by the local formal data $\Lambda_{t_i}$ at $t_i$ for $i=1,\ldots,\nu,\infty$.
\end{lem}

\begin{proof}Recall (see Section \ref{ss:LocalConsequences}) that the formal data at $x=t_i$ is determined
by the negative parts of the trace and determinant 
$$[\mathrm{tr}(\Omega_0)]^{<0}_{x=t_i}\ \ \ \text{and}\ \ \ [\det(\Omega_0)]^{<-n}_{x=t_i}$$
which are determined by the negative part 
\begin{equation}\label{eq:LocGaugeOmega2}
[\Omega_0]^{<0}_{x=t_i}=\begin{pmatrix}0 & b_i\\ C_i(x) & D_i(x)\end{pmatrix}\frac{dx}{(x-t_i)^{n_i}}\ \ \ \text{where}\ \ \ 
b_i=\frac{1}{\prod_{j\not=i}(t_i-t_j)^{n_j}}.
\end{equation}
We easily deduce that 
\begin{equation}\label{eq:CiDiFormalData}
[\mathrm{tr}(\Omega_0)]^{<0}_{x=t_i}=D_i(x)\frac{dx}{(x-t_i)^{n_i}}\ \ \ \text{and}\ \ \ 
[\det(\Omega_0)]^{<-n}_{x=t_i}=-b_i C_i(x)\left( \frac{dx}{(x-t_i)^{n_i}}\right).
\end{equation}
Therefore, $C_i$ and $D_i$ are uniquely determined by the formal data.

Finally, the negative part at $x=\infty$ ($\Leftrightarrow \zeta=0$) is given by 
\begin{equation}\label{eq:LocGaugeOmega2Infty}
[\tilde\Omega_0]^{<0}_{\zeta=0}\ =\ -\begin{pmatrix}0 & \left[\frac{1}{\frac{1}{\zeta}P\left(\frac{1}{\zeta}\right)}\right]^{<0}_{\zeta=0} \\ \frac{1}{\zeta}C_\infty\left(\frac{1}{\zeta}\right) & 
\frac{1}{\zeta^2}D_\infty\left(\frac{1}{\zeta}\right)+\frac{1+\sum_{i=1}^\nu\tau_i}{\zeta}\end{pmatrix}d\zeta
\end{equation}
where $\tau_i=\mathrm{Res}_{x=t_i}\frac{D_i(x)}{(x-t_i)^{n_i}}dx$ is the residue of the $(2,2)$-coefficient at $t_i$.
Note that the trace of residue at $x=\infty$ has to be $1+\sum_{i=1}^\nu\tau_i$ by Fuchs relation 
(see Proposition \ref{prop:Fuchs}) for the initial connection
$\nabla$ on $E=\mathcal O\oplus\mathcal O(1)$. Again, we can determine $C_\infty,D_\infty$ in terms of local formal
data $\Lambda_\infty$ as in formula (\ref{eq:CiDiFormalData}).
\end{proof}

\begin{remark}We have just proved that the principal part of the matrix connection at essential poles $t_i$'s
depends only on local formal data $\Lambda$. It is independant of the choice of the connection in the moduli
space $\Con_{\Mu}^\Lambda$. Moreover, the local Hukuhara-Levelt-Turritti decomposition at poles $t_i$'s
is fixed independly of the connection up to order $n_i-1$ as proved in the following. 
\end{remark}

\begin{cor}\label{cor:FixedLeveltDec}
The local (possibly ramified) formal decomposition $E=L^+\oplus L^-$ by $\nabla_2$-invariant subbundles $L^{\pm}$
at each singular point $t_i$ is determined up to order $n_i-1$ by the following equation
$$J^{n_i-1}\bigg(L^{\pm}\bigg)=\begin{pmatrix}1\\ J^{n_i-1}(y)\end{pmatrix}\ \ \ \text{where}\ \ \ y^2=\left(\prod_{j\not=i}(x-t_j)^{n_j}\right)\bigg(C_i(x)+yD_i(x)\bigg).$$
In particular, it depends only on formal invariants, not of accessory parameters.
\end{cor}

\begin{proof}Straightforward computation with the negative part of (\ref{eq:GlobalNormalForm}) which is given 
by (\ref{eq:LocGaugeOmega2}) and  (\ref{eq:LocGaugeOmega2Infty}).
\end{proof}

\begin{lem}\label{Lem:tildeC}
The polynomial $\tilde C$ occuring in formula (\ref{eq:GlobalNormalForm}) is uniquely determined
by the fact that all singular points 
$x=q_j$ are apparent (i.e. not logarithmic).
\end{lem}

\begin{proof}For each $x=q_j$, we can write
$$\Omega_0=\begin{pmatrix}0 & \frac{1}{P(x)} \\ \frac{\pp_j}{x-q_j}+\tilde C(x)+\hat{c}_j(x) & -\frac{1}{x-q_j}+\hat{d}_j(x)\end{pmatrix}dx$$
where $\hat{c}_j,\hat{d}_j$ are holomorphic at $x=q_j$. The singular point $x=q_j$ is apparent if, and only if, it disappears
after elementary transformation $Y_0=M\tilde Y_0$ with 
$$M=\begin{pmatrix}1& 0\\ \pp_j & x-q_j\end{pmatrix}.$$
We can check that the new matrix connection is holomorphic if, and only if, 
$$\tilde C(q_j)+\hat{c}_j(q_j)+\pp_j\hat{d}_j(q_j)-\frac{\pp_j^2}{P(q_j)}=0.$$
Running over the $n-3$ distinct points $q_j$, we get $n-3$ independant relations on $\tilde C(x)$ which is therefore 
uniquely determined.
\end{proof}

We can now conclude the proof of our results.

\begin{proof}[Proof of Lemma \ref{lem:UniversalNabla0}, Proposition \ref{prop:AppIsom} and Theorem \ref{thm:BiratModel}.]
Given the polar divisor $D$, the formal data $\Lambda$ and $\{\mathbf p_1,\ldots,\mathbf p_{(n-3)}\}\in\W \subset \mathrm{Hilb}^{(n-3)}(\OMD)$, we can reconstruct $\nabla_0$ on $E_0$ uniquely satisfying all properties of Lemma \ref{lem:UniversalNabla0}
as follows. First of all, the matrix $\Omega_0$ of $\nabla_0$ must be in companion form (\ref{eq:Omega2}) (item 2 of the Lemma)
and coefficients $c_0,d_0$ of the form (\ref{eq:GlobalNormalForm}) (items 1 and 4 of the Lemma).
By Lemmae \ref{lem:CiDi} and \ref{Lem:tildeC}, the polynomials $C_i,D_i$ and $\tilde C$ are determined 
by local formal data $\Lambda_{t_i}$ and the fact that all $x=q_j$ are apparent (items 3 and 4 of Lemma).
Finally, $\mathbf p_j=(q_j,\pp_j)$ determines $q_j$ and  $\pp_j$ (item 4) and Lemma \ref{lem:UniversalNabla0} is proved.

After performing a negative elementary transformation of the vector bundle $E_0$ at each point $q_j$
in order to get rid of these apparent poles, we get a new connection $(E,\nabla)$ which has
the property that $\Phi(E,\nabla)=\{\mathbf p_1,\ldots,\mathbf p_{(n-3)}\}$. Therefore, $\Phi$
admits an inverse in restriction to $\V\leftrightarrow\W$, proving the second assertion 
of Proposition \ref{prop:AppIsom}, and therefore Theorem \ref{thm:BiratModel}.
\end{proof}

\section{Symplectic structure and Lagrangian fibration}

We have identified the open set $\V\subset\Con_{\Mu}^\Lambda$ of those connections $(E,\nabla)$ 
with $E=\mathcal O\oplus\mathcal O(1)$ with $q_j\not=t_i$ and $q_j\not=q_k$, $j\not=k$, with the open 
set $\W\subset\mathrm{Hilb}^{(n-3)}(\OMD)$, also defined by the same 
restrictions on $q_j$ (see introduction).

On the other hand, this latter space has dimension $2(n-3)$ and can be equipped with a meromorphic 
$2$-form $\omega$ inducing a symplectic structure on a large open set including $\Pi^{-1}(W)$.
Indeed, the natural map  which to a (local holomorphic) section 
$\zeta(x)\cdot\frac{dx}{P(x)}$ of $\Omega^1(D)$ associates the meromorphic section $\frac{\zeta(x)}{P(x)}\cdot dx$
of $\Omega^1$ induces a rational map $\psi:\OMD\dashrightarrow \OM0$ between the total spaces.
The Liouville symplectic form $\omega_{Liouv}$ on $\OM0$ induces by pull-back a $2$-form
$\omega:=\psi^*\omega_{Liouv}$.
In local coordinates $(x,\zeta)$ on $\OMD$ where $\zeta$ stands for the section $\zeta\cdot\frac{dx}{P(x)}$,
we have $\omega=d\zeta\wedge\frac{dx}{P(x)}$. The $2$-form $\sum_{j=1}^{n-3}d\zeta_j\wedge\frac{dx_j}{P(x_j)}$
on the product $\OMD^{(n-3)}$ is obviously invariant under permutation of factors
and defines a rational $2$-form on the quotient $\mathrm{Sym}^{(n-3)}\OMD$
with poles along hypersurfaces $t_i=q_j$, and therefore on the corresponding open set $\W\subset\mathrm{Hilb}^{(n-3)}\OMD$. Finally, the symplectic form induced on $\V\subset\Con_{\Mu}^\Lambda$
via $\Phi$ has the explicit expression
\begin{equation}\label{eq:SymplecticForm}
\omega=\sum_{j=1}^{n-3}dp_j\wedge dq_j\ \ \ \text{where}\ \ \ p_j:=\frac{\pp_j}{P(q_j)}
\end{equation}
Recently, Komyo proved in \cite{Komyo} that it coincides with the natural symplectic form
on moduli spaces defined by Attyiah-Bott, and generalized in the irregular case by Boalch in \cite{Boalch}
and by Inaba and Saito in \cite{InabaSaito,Inaba}.
In the next section, we check this in the Fuchsian case where everything is known to be explicit.

\section{Link with scalar equation}\label{sec:Scalar}

Our normal form (\ref{eq:GlobalNormalForm}) almost looks like a companion matrix. After setting 
$$Y_2=MY_3\ \ \ \text{with}\ \ \ M=\begin{pmatrix}1&0\\ 0 & -P(x)\end{pmatrix},$$
we obtain the matrix connection in companion form
$$\Omega_3=\begin{pmatrix}0 &-1 \\ -\frac{c_2}{P} & d_2+\frac{P'}{P}\end{pmatrix}$$
that corresponds to the scalar differential equation
\begin{equation}\label{eq:scalar}
u''+\left(d_2+\frac{P'}{P}\right)u'+\left(-\frac{c_2}{P}\right)u=0.
\end{equation}
Starting from Okamoto's Fuchsian equation for Garnier system (see \cite[formula (0.8)]{Kimura})
$$u''+\left(\frac{1-\kappa_0}{x}+\frac{1-\kappa_1}{x-1}+\sum_{i=1}^{n-3}\frac{1-\theta_i}{x-t_i}
-\sum_{j=1}^{n-3}\frac{1}{x-q_j}\right)u'$$
$$+\left(\frac{\kappa}{x(x-1)}-\sum_{i=1}^{n-3}\frac{t_i(t_i-1)K_i}{x(x-1)(x-t_i)}
+\sum_{j=1}^{n-3}\frac{q_j(q_j-1)\mu_j}{x(x-1)(x-q_j)}\right)u=0$$
with $K_i$ determined by \cite[formula (0.9)]{Kimura}, we get the normal form 
$$\Omega_2= \begin{pmatrix}0 & \frac{1}{P(x)} \\ c_2(x) & d_2(x)\end{pmatrix}dx$$
with 
$$\left\{\begin{matrix}
c_2(x)=-\sum_{j=1}^{n-3}\frac{P(q_j)\mu_j}{x-q_j}+C(x),\hfill\hfill \\
d_2(x)=-\frac{\kappa_0}{x}-\frac{\kappa_1}{x-1}-\sum_{i=1}^{n-3}\frac{\theta_i}{x-t_i}
-\sum_{j=1}^{n-3}\frac{1}{x-q_j}
\end{matrix}\right.$$
where $C(x)$ is a polynomial of degree $\le n-3$ and $P(x)=x(x-1)\prod_{i=1}^{n-3}(x-t_i)$.
The symplectic form $\omega_0=\sum_{j=1}^{n-3}d\mu_j\wedge d\lambda_j$ in \cite[page 47]{Kimura} (here $\lambda_j=q_j$)
coincide up to a sign with the natural symplectic form of our normal form
$\omega=\sum_{j=1}^{n-3}dp_j\wedge dq_j$. By a direct computation, one easily check that 
$\omega$ is also the symplectic form in \cite{DM} since their canonical coordinates
are given by 
$$q_j+\frac{1-\kappa_0}{q_j}+\frac{1-\kappa_1}{q_j-1}+\sum_{i=1}^{n-3}\frac{1-\theta_i}{q_j-t_i}.$$
We expect it is also possible to generalize the Hamiltonian system described in \cite{DM}
to the irregular case following our approach.

Starting from Kimura's most degenerate scalar equation (see \cite[L(9/2;2) page 37]{Kimura})
$$u''+\left(-\sum_{j=1}^{2}\frac{1}{x-q_j}\right)u'
+\left(-9x^5-9t_1x^3-3t_2x^2-3K_2x-3K_1
+\sum_{j=1}^{2}\frac{\mu_j}{(x-q_j)}\right)u=0$$
with $K_1,K_2$ determined in \cite[H(9/2) page 40]{Kimura}, we get the normal form (\ref{eq:GlobalNormalForm}) with
$$\left\{\begin{matrix}
c_2(x)=9x^5+9t_1x^3+3t_2x^2+3K_2x+3K_1
-\sum_{j=1}^{2}\frac{\mu_j}{(x-q_j)}, \\
d_2(x)=-\sum_{j=1}^{2}\frac{1}{x-q_j}\hfill\hfill
\end{matrix}\right.$$
and again we find the same symplectic structure up to a sign: $\omega=\sum_{j=1}^2dp_j\wedge dq_j$, $p_j=-\mu_j$.

\bibliographystyle{amsplain}

\end{document}